\documentclass{article}

\usepackage{amsfonts}
\usepackage{amsthm}
\usepackage{amsmath}
\usepackage{amssymb}
\usepackage{mathtools} %for prescript
\usepackage{authblk} %for affil
\usepackage{xcolor}

\usepackage[top=1in, bottom=1in, left=1in, right=1in]{geometry}

\newtheorem{theorem}{Theorem}[section]
\newtheorem{lemma}{Lemma}[section]

\newtheorem{proposition}{Proposition}[section]

\theoremstyle{definition}
\newtheorem{definition}{Definition}[section]

\theoremstyle{remark}
\newtheorem{remark}{Remark}[section]
\newtheorem{example}{Example}[section]

\begin{document}

\title{On a new class of fractional difference-sum operators based on discrete Atangana--Baleanu sums}
% Force line breaks with \\

\date{}
 
\author[1]{Thabet Abdeljawad}
\author[2,3]{Arran Fernandez}

\affil[1]{{\small Department of Mathematics and General Sciences, Prince Sultan University, P. O. Box 66833, Riyadh 11586, Saudi Arabia}}
\affil[2]{{\small Department of Applied Mathematics and Theoretical Physics, University of Cambridge, Wilberforce Road, Cambridge, CB3 0WA, Cambridge, United Kingdom}}
\affil[3]{{\small Department of Mathematics, Faculty of Arts and Sciences, Eastern Mediterranean University, Famagusta, TRNC, Mersin-10, Turkey}}

\maketitle


\begin{abstract}
We formulate a new class of fractional difference and sum operators, study their fundamental properties, and find their discrete Laplace transforms. The method depends on iterating the fractional sum operators corresponding to fractional differences with discrete Mittag-Leffler kernels. The iteration process depends on the binomial theorem. We note in particular the fact that the iterated fractional sums have a certain semigroup property and hence the new introduced iterated fractional difference-sum operators have this semigroup property as well.
\end{abstract}

%\begin{quotation}
%Fractional-order differences and sums form the basis of discrete fractional calculus, a field which combines difference calculus and fractional calculus. There are many different definitions for (both discrete and continuous) fractional-order operators, and one which has emerged strongly in recent years with many applications is the Atangana--Baleanu model. In the continuous case this model is based on integrals with non-singular Mittag-Leffler kernels, and this paper is concerned with the discrete analogues of such operators. It has been shown that the Atangana--Baleanu model fails to satisfy the natural semigroup property, so here we propose a related operator, based on the concept of fractional iteration, which retains all the useful properties of the Atangana--Baleanu difference calculus but also possesses a semigroup property. Fundamental properties of this new operator are established, and some difference equations are solved using it.
%\end{quotation}

\section{Introduction and preliminaries} \label{s:1}

Discrete fractional calculus is an important emerging branch of analysis \cite{Miller,Goodrich,holm,Ferd,Feri,Suwan,Thbinomial}, which has been very useful in the analysis of discrete systems with non-local effects. The solution of discrete fractional differential equations and discrete boundary value problems has discovered many applications \cite{Atmodel}, and so this is an important field to develop in terms of mathematical theory.

There are many different types of fractional calculus which can be defined, in both the continuous and discrete contexts. The point of such an exercise is to discover new ways of modelling various fractional systems, and to create new frameworks which can then be used in a number of applications \cite{dumitru,Martin}. It is important to continue developing these new models, from several points of view.

In the last two years, some important new definitions of fractional calculus have been formulated with exponential and Mittag-Leffler kernels \cite{FCaputo,Abdon}, and their properties have been explored in a number of papers \cite{baleanu-mousalou-rezapour1,baleanu-fernandez,al-refai,fernandez-baleanu1}. Furthermore, discrete fractional calculus has been theoretically developed more by formulating and analysing discrete versions of these fractional operators \cite{abdeljawad-baleanu,TQ CAM 2018,TD ADE 2016}. The idea of our approach in this article depends on iterating the fractional sums corresponding to fractional operators with discrete Mittag-Leffler kernels, extending and generalising these operators in such a way as to recover certain desirable traits such as a semigroup property. The idea of considering iterations of functional operators to get fractional ones is not new -- indeed, it is the very basis of fractional calculus itself \cite{miller-ross,oldham-spanier} -- but it has recently been explored as a way of adding further fractionalisation to operators which are already considered as fractional \cite{jarad,fernandez-baleanu,fernandez-baleanu-srivastava}. The present work can be seen as an extension of these projects into the discrete context. The advantage of the discrete operators compared to the existing ones is the same as in any branch of discrete calculus: it is useful to have definitions in both discrete and continuous contexts, since modelling different processes in the real world requires both discrete and continuous models \cite{Goodrich,holm,erbe-goodrich-jia-peterson,atici-uyanik,herrmann}.

The structure of this work is as follows. In the current Section \ref{s:1}, we review some basic concepts about discrete fractional calculus in the frame of nabla difference analysis, including the well-known Atangana--Baleanu (AB) model of discrete fractional calculus and its fundamental properties. In Section \ref{s:2}, we define our new family of operators and analyse them, proving some essential facts about them. In Section \ref{s:3}, we consider some fractional difference equations in this new model. Finally, in Section \ref{s:4} we conclude the paper.

\subsection{Nabla discrete calculus}

\begin{definition} \label{rising}
(i) For any $l\in\mathbb{N}$ and any number $z$, the $l$ rising factorial of $z$ is
\begin{equation}\label{rising 1}
    z^{\overline{l}}= \prod_{i=0}^{l-1}(z+i),~~~z^{\overline{0}}=1.
\end{equation}

(ii) For any $\mu\in\mathbb{R}$, the $\mu$ rising function is
\begin{equation}\label{alpharising}
 z^{\overline{\mu}}=\frac{\Gamma(z+\mu)}{\Gamma(z)},~~~z \in \mathbb{R}\setminus \{...,-2,-1,0\},~~0^{\overline{\mu}}=0
\end{equation}

\end{definition}

The following fact is straightforward to prove:
 \begin{equation}\label{oper}
    \nabla (z^{\overline{\mu}})=\mu z^{\overline{\mu-1}},
\end{equation}
and this means $z^{\overline{\mu}}$ is an increasing function on $\mathbb{N}_0$.

\indent

\begin{definition}[Nabla fractional sums -- see \cite{dualCaputo, dualR}]
\label{fractional sums}
Define the operator $\rho(t)=t-1$, which is called the backwards jump. For any function $f:\mathbb{N}_a:=\{a,a+1,a+2,..\}\rightarrow \mathbb{R}$,
the nabla fractional sum of order $\mu>0$ and of left type starting from $a$ is defined by:
\begin{equation*}
\prescript{}{a}\nabla^{-\mu} f(z)=\frac{1}{\Gamma(\mu)}
\sum_{s=a+1}^z(z-\rho(s))^{\overline{\mu-1}}f(s),
\quad z \in \mathbb{N}_{a+1}.
\end{equation*}
Similarly, for a function $f:~_{b}\mathbb{N}:=\{b,b-1,b-2,..\}\rightarrow \mathbb{R}$, the nabla fractional sum of order $\mu>0$ and of right type finishing at $b$ is defined by:
\begin{align*}
 \nabla_b^{-\mu} f(z) &= \frac{1}{\Gamma(\mu)}\sum_{s=z}^{b-1}(s-\rho(z))^{\overline{\mu-1}}f(s) \\
   &= \frac{1}{\Gamma(\mu)}\sum_{s=z}^{b-1}(\sigma(s)-z)^{\overline{\mu-1}}f(s),\quad z \in {_{b-1}\mathbb{N}}.
\end{align*}

\end{definition}

\begin{lemma}[\cite{dualR,ThFer}]  \label{power nabla left and right}
Let $\alpha>0,~\beta>-1,~h>0$. Then we have the following identities:
\begin{align}
\label{pdl1} \prescript{}{a}\nabla^{-\alpha} (t-a)^{\overline{\beta}} &= \frac{\Gamma(\beta+1)}{\Gamma(\beta+1+\alpha)}(t-a)^{\overline{\alpha+\beta}} \\
\label{pdr111} \nabla_{b}^{-\alpha} (b-t)^{\overline{\beta}} &= \frac{\Gamma(\beta+1)}{\Gamma(\beta+1+\alpha)}(b-t)^{\overline{\alpha+\beta}}
\end{align}

\end{lemma}
\indent

From \cite{Goodrich,dualR}, we recall that the left and right nabla fractional sums satisfy the following semigroup property:
\begin{align}
\label{semi nablal} \prescript{}{a}\nabla^{-\alpha}\prescript{}{a}\nabla^{-\mu}f(v)&=\prescript{}{a}\nabla^{-(\alpha+\mu)}f(v), \\
\label{semi nablar} \nabla_b^{-\alpha}~\nabla_b^{-\mu}f(v)&=\nabla_b^{-(\alpha+\mu)}f(v),
\end{align}

\begin{definition}[Nabla discrete Laplace transforms -- see \cite{TD AIDE 2016}]\label{dl}
The nabla discrete Laplace transform $\mathcal{K}=\mathcal{K}_0$, applied to a function $f$ defined on $\mathbb{N}_0$, is defined by
\begin{equation}\label{dle}
   \mathcal{K} f(z)=\sum_{t=1}^\infty (1-z)^{t-1}f(t).
\end{equation}
More generally, for any $a$, if $f$ is a function on $\mathbb{N}_a$, the nabla discrete Laplace transform $\mathcal{K}_a$ is defined by
\begin{equation}\label{gdle}
   \mathcal{K}_a f(z)=\sum_{t=a+1}^\infty (1-z)^{t-a-1}f(t).
\end{equation}
\end{definition}

\begin{lemma} \label{At}
 For any $\mu \in \mathbb{R}\setminus\{...,-2,-1,0\}$, we have the following results on nabla discrete Laplace transforms.
 \begin{itemize}
   \item[(i)] $\mathcal{K}(t^{\overline{{\mu-1}}})(z)=\frac{\Gamma(\mu)} {z^\mu},~~|1-z|<1$,
   \item[(ii)] $\mathcal{K}(t^{\overline{{\mu-1}}}b^{-t})(z)=\frac{b^{\mu-1}\Gamma(\mu)}{(z+b-1)^\mu},~~|1-z|<b$.
 \end{itemize}
\end{lemma}

\begin{proof}
See \cite{Nabla}.
\end{proof}

\begin{remark}\label{rem1}
We can extend (i) of Lemma \ref{At} to the more general statement that
 \[(\mathcal{K}_{a}(t-a)^{\overline{\mu-1}})(s)= \frac{\Gamma(\mu)}{s^\mu}.\]
\end{remark}

\begin{definition} \label{nDML}[Nabla Discrete Mittag-Leffler -- see \cite{dualCaputo, dualR, Thsemi,TD AIDE 2016}] For $\lambda \in \mathbb{R}$ with $|\lambda|<1$ and $\alpha, \beta,\rho, v \in \mathbb{C}$ with $Re(\alpha)>0$, the nabla discrete  Mittag-Leffler functions with one, two, and three parameters are defined respectively by:

\begin{align}
\label{nM22} E_{\overline{\alpha}} (\lambda, v)&=  \sum_{k=0}^\infty \lambda^k
\frac{v^{\overline{k\alpha}}} {\Gamma(\alpha k+1)}; \\
\label{ML2} E_{\overline{\alpha, \beta}}(\lambda,v)&= \sum_{k=0}^\infty \lambda^k
\frac{v^{\overline{k\alpha+\beta-1}}} {\Gamma(\alpha k+\beta)}; \\
\label{ML3}  E^\rho_{\overline{\alpha,\beta}} (\lambda, v)&=\sum_{k=0}^\infty  (\rho)_k
\frac{v^k} { k! \Gamma(\alpha k+\beta)}.
\end{align}

\end{definition}

Note that we have $E_{\overline{\alpha}} (\lambda, v)=E_{\overline{\alpha, 1}}(\lambda,v)$ and
$E_{\overline{\alpha,\beta}}(\lambda,v)=E^1_{\overline{\alpha,\beta}} ( \lambda,v)$, just as in the continuous case \cite{Mainardi}.

\begin{proposition} \label{sum and diff}For any $\lambda,\alpha, \beta,\rho, v \in \mathbb{C}$ as in Definition \ref{nDML} and $\gamma \in \mathbb{C}$ with $Re(\gamma)>0$, we have the following difference and summation properties of discrete Mittag-Leffler functions.

\begin{align*}
\nabla_v E_{\overline{\alpha}} (\lambda, v)&= \lambda E_{\overline{\alpha,\alpha}} (\lambda, v); \\
\nabla_v E^\rho_{\overline{\alpha,\beta}} (\lambda, v)&=E^\rho_{\overline{\alpha,\beta-1}} (\lambda, v); \\
\sum_{t=a+1}^v  E_{\overline{\alpha,\beta}} (\lambda, t-a) &=E_{\overline{\alpha,\beta+1}} (\lambda, v-a); \\
\prescript{}{a}\nabla^{-\gamma}  E^\rho_{\overline{\alpha,\beta}} (\lambda, v-a)&=E^\rho_{\overline{\alpha,\beta+\gamma}} (\lambda, v-a).
\end{align*}

\end{proposition}

\begin{proof}The proof is straightforward  (see \cite{TD AIDE 2016}). In the proof of the last part, the assertion \cite{ThFer,dualCaputo,dualR,Goodrich} that \[\prescript{}{a}\nabla^{-\gamma} (t-a)^{\overline{\nu-1}}= \frac{\Gamma(\nu)}{\Gamma(\nu+\gamma)}(t-a)^{\overline{\gamma+\nu-1}}\] has been used. Also, notice that the second part is the particular case $\gamma=-1$ of the last part.
\end{proof}

The remainder of this section is dedicated to summarising some known results about discrete Laplace transforms for Mittag-Leffler functions and functions of convolution type. More details can be found in \cite{Thsemi}.

\begin{definition}[Nabla discrete convolutions -- see  \cite{Thsemi,Goodrich}]\label{conv}
Let $a \in \mathbb{R}$ and consider two functions $f,g:\mathbb{N}_a\rightarrow \mathbb{R}$. Their nabla discrete convolution is
\begin{equation}\label{dconv}
   ( f\ast g)(v)=\sum_{s=a+1}^v g(v-\rho(s)+a) f(s).
\end{equation}
\end{definition}

\begin{proposition}[\cite{Thsemi,Goodrich}] \label{convprop}
For any $a \in \mathbb{R}$  and functions $f,g$ defined on $\mathbb{N}_a$, we have the following convolution property of Laplace transforms in the nabla discrete context:
\begin{equation}\label{conv1}
    (\mathcal{K}_a (f \ast g))(s)=  (\mathcal{K}_af)(s) (\mathcal{K}_a g)(s).
\end{equation}
\end{proposition}

\begin{lemma}[\cite{Thsemi}] \label{lap of nabla}
Let $a\in\mathbb{R}$ and let $f$ be a function defined on $\mathbb{N}_a$. Then
\begin{equation}\label{lap of 1}
    (\mathcal{K}_a \nabla(f(t))(s)=s (\mathcal{K}_af)(s)-f(a).
\end{equation}
\end{lemma}

\begin{lemma}\cite{Nabla} \label{F}
For any $a\in\mathbb{R}$ and $\nu\in\mathbb{R}^+$, we have
\[(\mathcal{K}_{a} ~_{a}\nabla^{-\nu})f(s)=s^{-\nu}(\mathcal{K}_a f)(s).\]
\end{lemma}

\begin{lemma} \label{compute certain functions}\cite{Thsemi}
Let $0<\alpha \leq 1$, $a\in\mathbb{R}$, and $f$ be a function defined on $\mathbb{N}_a$. Then:
\begin{itemize}
\item[(i)] $(\mathcal{K}_a E_{\overline{\alpha}}(\lambda,t-a))(z)= \frac{z^{\alpha-1}}{z^\alpha-\lambda}.$
\item[(ii)] $(\mathcal{K}_a E_{\overline{\alpha,\alpha}}(\lambda,t-a))(z)= \frac{1}{z^\alpha-\lambda}.$
\end{itemize}
\end{lemma}

\subsection{Discrete Atangana--Baleanu fractional differences}
Let us review the basic theory of fractional sums and differences defined using discrete Mittag-Leffler kernels, as presented in \cite{TD AIDE 2016} based on the original ideas in \cite{Mainardi, Antony etal,Abdon}.

\begin{definition}[\cite{TD AIDE 2016}] \label{dABDdef}
Let $\alpha \in [0,1]$ and $a<b$ in $\mathbb{R}$. For a function $f$ defined on $\mathbb{N}_a$,  its nabla discrete AB left fractional difference is defined in Caputo type by
\begin{equation}\label{d1}
\left(\prescript{ABC}{a}\nabla^\alpha f\right)(t)=\frac{B(\alpha)}{1-\alpha} \sum_{s=a+1}^t\nabla_s f(s)E_{\overline{\alpha}}\left(\frac{ -\alpha}{1-\alpha}, t-\rho(s) \right),
\end{equation}
and in Riemann-Liouville type by
\begin{equation}\label{d2}
\left(\prescript{ABR}{a}\nabla^\alpha f\right)(t)=\frac{B(\alpha)}{1-\alpha}\nabla_t \sum_{s=a+1}^t f(s)E_{\overline{\alpha}}\left(\frac{ -\alpha}{1-\alpha}, t-\rho(s) \right).
\end{equation}
Similarly, for a function $f$ defined on $_b\mathbb{N}$, its nabla discrete AB right fractional difference is defined in Caputo type by
\begin{equation}\label{Crd}
 \left(\prescript{ABC}{}\nabla_b^\alpha f\right)(t)=\frac{B(\alpha)}{1-\alpha} \sum_{s=t}^{b-1} (-\Delta_s f)(s)E_{\overline{\alpha}}\left(\frac{ -\alpha}{1-\alpha}, s-\rho(t)\right).
\end{equation}
and in Riemann--Liouville type by
\begin{equation}\label{nrd}
 \left(\prescript{ABR}{}\nabla_b^\alpha f\right)(t)=\frac{B(\alpha)}{1-\alpha}(-\Delta_t) \sum_{s=t}^{b-1} f(s)E_{\overline{\alpha}}\left(\frac{ -\alpha}{1-\alpha}, s-\rho(t)\right).
\end{equation}
The notations ABR and ABC are used to denote AB fractional differences of Riemann--Liouville and Caputo type respectively.

Note that since the discrete Mittag-Leffler kernel \eqref{nM22} converges for $|\lambda|<1$, and in this case $|\lambda|=\frac{\alpha}{1-\alpha}$, the kernels in all four of the above definitions are convergent for $0<\alpha <\frac{1}{2}$.
\end{definition}

We now define the fractional sums corresponding to the fractional difference defined in Definition \ref{dABDdef}. Again this is analogous to the definition found in \cite{Abdon} for the continuous case.

\begin{definition}\cite{TD AIDE 2016} \label{dABIdef}
For $0<\alpha<1$ and a function $f$ defined on $\mathbb{N}_a$, the left fractional sum of AB type is
\begin{equation}\label{T4}
  \left(\prescript{AB}{a}\nabla^{-\alpha} f\right)(t)=\frac{1-\alpha}{B(\alpha)}f(t)+\frac{\alpha}{B(\alpha)} \left(\prescript{}{a}\nabla^{-\alpha}f\right)(t).
\end{equation}
Similarly, for a function $f$ defined on $~_{b}\mathbb{N}$, the right fractional sum of AB type is
\begin{equation}\label{nrd}
  \left(\prescript{AB}{}\nabla_b^{-\alpha} f\right)(t)=\frac{1-\alpha}{B(\alpha)}f(t)+\frac{\alpha}{B(\alpha)} \left(\nabla_{b}^{-\alpha} f\right)(t).
\end{equation}
\end{definition}

\begin{theorem}
For any $\alpha\in(0,\frac{1}{2})$ and any function $f$ defined on $\mathbb{N}_a\cap~_b\mathbb{N}$, we have the following relations between the AB fractional differences of Caputo and Riemann--Liouville type and the associated AB fractional sum.
\begin{align*}
\left(\prescript{ABR}{a}\nabla^{\alpha}\prescript{AB}{a}\nabla^{-\alpha} f\right)(t)&=f(t); \\
\left(\prescript{AB}{a}\nabla^{-\alpha}\prescript{ABR}{a}\nabla^{\alpha} f\right)(t)&=f(t); \\
\left(\prescript{ABR}{}\nabla_b^\alpha \prescript{AB}{}\nabla_b^{-\alpha }f\right)(t)&=f(t); \\
\left(\prescript{AB}{}\nabla_b^{-\alpha} \prescript{ABR}{}\nabla_b^{\alpha }f\right)(t)&=f(t); \\
\left(\prescript{ABC}{a}\nabla^{\alpha} f\right)(t)=\left(\prescript{ABR}{a}\nabla^{\alpha} f\right)(t)&-f(a)\frac{B(\alpha)}{1-\alpha}E_{\overline{\alpha}}(\lambda,t-a); \\
\left(~^{ABC}\nabla_b^{\alpha} f\right)(t)=\left(~^{ABR}\nabla_b^{\alpha} f\right)(t)&-f(b)\frac{B(\alpha)}{1-\alpha}E_{\overline{\alpha}}(\lambda,b-t).
\end{align*}
\end{theorem}

\begin{proof}
See \cite{TD AIDE 2016}.
\end{proof}

The following lemma, the discrete analogue of a result proved for the continuous AB model in \cite{baleanu-fernandez}, is essential to proceed in confirming our representations.

\begin{lemma}[\cite{TQ CAM 2018}]\label{TQ}
For any $0<\alpha< \frac{1}{2}$, with $\lambda:= \frac{-\alpha}{1-\alpha}$ and $f$ being a function defined on $\mathbb{N}_a$, we have
\begin{equation}\label{bbb}
 \left(\prescript{ABR}{a}\nabla^{\alpha} f\right)(t)=\frac{B(\alpha)}{1-\alpha}\left[f(t)+\sum_{k=1}^\infty \lambda^k (\prescript{}{a}\nabla^{-\alpha k} f)(t)\right].
\end{equation}
\end{lemma}

\section{Iterated AB fractional difference-sum operators } \label{s:2}
\subsection{Definitions}
In this section,  by iterating the AB fractional sums we shall formulate a new class of fractional difference-sum operators, which has a semigroup property in one of its two parameters. These operators are the discrete analogue of the iterated AB differintegrals defined in \cite{fernandez-baleanu}.

%-----------------------------------------------------------------------------------------------------------------%

Consider the AB left fractional sum defined in \eqref{T4}. If we iterate this operator $n$ times using the binomial theorem, and make use of the semigroup property \eqref{semi nablal} and the fact that $~_{a}\nabla^{-0}f(t)=f(t)$, then we have:
\begin{align*}%\label{derive}
\left(\prescript{AB}{a}\nabla^{-\alpha}\right)^n f(t) &= \left[\frac{1-\alpha}{B(\alpha)}+\frac{\alpha}{B(\alpha)}\prescript{}{a}\nabla^{-\alpha}\right]^n f(t) \\
   &= \sum_{k=0}^n \frac{\binom{n}{k}(1-\alpha)^{n-k} \alpha^k}{B(\alpha)^n}\prescript{}{a}\nabla^{-k\alpha}f(t) \\
   &= \left(\frac{1-\alpha}{B(\alpha)}\right)^n f(t) \\
   &\hspace{2cm}+\sum_{k=1}^n \frac{\binom{n}{k}(1-\alpha)^{n-k} \alpha^k}{B(\alpha)^n \Gamma(\alpha k)}\left(\sum_{s=a+1}^t (t-\rho(s))^{\overline{k\alpha-1}}f(s)\right) \\
   &= \sum_{s=a+1}^t f(s) \Bigg[ \left(\frac{1-\alpha}{B(\alpha)}\right)^n \delta(t-\rho(s)) \\
   &\hspace{4cm}+ \sum_{k=1}^n \frac{\binom{n}{k}(1-\alpha)^{n-k} \alpha^k}{B(\alpha)^n \Gamma(\alpha k)} (t-\rho(s))^{\overline{k\alpha-1}}\Bigg],
\end{align*}
where $\delta(t-\rho(s))$ is the Dirac delta function, namely the function defined on the time scale $\mathbb{N}$ by
   \begin{equation}\label{dirac}
    \delta(t-\rho(s))=\begin{cases}
                          0\quad\quad \text{ if } t\neq s, \\
                          1\quad\quad \text{ if } t=s.
                        \end{cases}
   \end{equation}

%-------------------------------------------------------------------------------------------------------------------%

   More generally, if $\mathbb{T}$ is an arbitrary time scale, then the Dirac delta function is defined by
    \begin{equation}\label{diracT}
    \delta_{\mathbb{T}}(t-\rho(s))=\begin{cases}
                          0\quad\quad\quad \text{ if } t\neq s, \\
                          \frac{1}{t-\rho_{\mathbb{T}}(t)}\quad \text{ if } t=s,
                        \end{cases}
   \end{equation}
   where $\rho_{\mathbb{T}}(t)$ is the backward jumping operator \cite{Martin} on the time scale $\mathbb{T}$ .
   In particular, if $\mathbb{T}=\mathbb{R}$, we have $t-\rho_{\mathbb{R}}(t)=0$ and hence we obtain the classical Dirac delta function.

   Now, for $\alpha\in [0,1]$, the above formulae for iteration of AB discrete operators can be fractionalised by replacing the index $n$ with a general $\mu  \in \mathbb{R}$ and replacing the finite sum by an infinite sum like the fractional binomial theorem. This idea is formalised in the following definition.

   \begin{definition} \label{Ab iterated}
   Let $\alpha \in [0,1]$, $\mu \in \mathbb{R}$, and $f$ be a function defined on $\mathbb{N}_a$. The iterated left AB fractional difference-sum of $f$ with order $(-\alpha,\mu)$, denoted by $~^{AB}_{a}\nabla^{(-\alpha,\mu)}f(t)$, is defined as:
   \begin{multline}\label{iterated AB def}
\prescript{AB}{a}\nabla^{(-\alpha,\mu)}f(t) =\sum_{k=0}^\infty \frac{\binom{\mu}{k}(1-\alpha)^{\mu-k} \alpha^k}{B(\alpha)^\mu}\prescript{}{a}\nabla^{-k\alpha}f(t)  \\
= \sum_{s=a+1}^t f(s) \left[ \left(\frac{1-\alpha}{B(\alpha)}\right)^\mu \delta(t-\rho(s))+ \sum_{k=1}^\infty \frac{\binom{\mu}{k}(1-\alpha)^{\mu-k} \alpha^k}{B(\alpha)^\mu \Gamma(\alpha k)} (t-\rho(s))^{\overline{k\alpha-1}}\right].
\end{multline}
Since $\mu$ can be negative, the left iterated AB difference of order $(-\alpha,-\mu)$ can be defined in exactly the same way, i.e.:
\begin{multline}\label{iterated AB diff}
\prescript{AB}{a}\nabla^{(-\alpha,-\mu)}f(t) =\sum_{k=0}^\infty \frac{\binom{-\mu}{k}B(\alpha)^\mu \alpha^k}{ (1-\alpha)^{\mu+k}}\prescript{}{a}\nabla^{-k\alpha}f(t)  \\
= \sum_{s=a+1}^t f(s) \left[ \left(\frac{B(\alpha)}{1-\alpha}\right)^\mu \delta(t-\rho(s))+ \sum_{k=1}^\infty \frac{\binom{-\mu}{k}B(\alpha)^\mu \alpha^k}{(1-\alpha)^{\mu+k} \Gamma(\alpha k)} (t-\rho(s))^{\overline{k\alpha-1}}\right].
\end{multline}
   \end{definition}

The iterated right AB fractional difference-sum and difference operators can be defined in an exactly analogous way to Definition \ref{Ab iterated}, namely as follows.
    \begin{definition} \label{rAb iterated}
   Let $\alpha \in [0,1]$, $\mu \in \mathbb{R}$, and $f$ be a function defined on $~_{b}\mathbb{N}$. The iterated right AB fractional difference-sum of $f$ with order $(-\alpha,\mu)$, denoted by $~^{AB}\nabla_b^{(-\alpha,\mu)}f(t)$, is defined as:
   \begin{multline}\label{iterated AB def}
\prescript{AB}{}\nabla_b^{(-\alpha,\mu)}f(t) =\sum_{k=0}^\infty \frac{\binom{\mu}{k}(1-\alpha)^{\mu-k} \alpha^k}{B(\alpha)^\mu}\nabla_b^{-k\alpha}f(t)  \\
= \sum_{s=t}^{b-1} f(s) \left[ \left(\frac{1-\alpha}{B(\alpha)}\right)^\mu \delta(s-\rho(t))+ \sum_{k=1}^\infty \frac{\binom{\mu}{k}(1-\alpha)^{\mu-k} \alpha^k}{B(\alpha)^\mu \Gamma(\alpha k)} (s-\rho(t))^{\overline{k\alpha-1}}\right].
\end{multline}
Since $\mu$ can be negative, the right iterated AB difference of order $(-\alpha,-\mu)$ can be defined in exactly the same way, i.e.:
   \begin{multline}\label{iterated AB diffr}
~^{AB}\nabla_b^{(-\alpha,-\mu)}f(t) =\sum_{k=0}^\infty \frac{\binom{-\mu}{k}B(\alpha)^\mu \alpha^k}{ (1-\alpha)^{\mu+k}}\nabla_b^{-k\alpha}f(t)  \\
= \sum_{s=t}^{b-1} f(s) \left[ \left(\frac{B(\alpha)}{1-\alpha}\right)^\mu \delta(s-\rho(t))+ \sum_{k=1}^\infty \frac{\binom{-\mu}{k}B(\alpha)^\mu \alpha^k}{(1-\alpha)^{\mu+k} \Gamma(\alpha k)} (s-\rho(t))^{\overline{k\alpha-1}}\right].
\end{multline}
   \end{definition}
   Notice that on the time scale $\mathbb{N}$ we have $s-\rho(t)=\sigma(s)-t$, and therefore the Dirac delta function in the sense of delta time scale analysis ($\delta^\Delta(\sigma(s)-t)$) is the same as $\delta(s-\rho(t))$.

   \begin{remark} Below are some special cases of the above Definitions \ref{Ab iterated} and \ref{rAb iterated}, which reflect their appropriateness as definitions of a two-parameter model of fractional calculus.

   \begin{itemize}
\item When $\mu=n\in\mathbb{N}$, we recover the expression derived above for iterating the AB fractional sum $n$ times. In particular, when $\mu=1$ we have
     \begin{align*}
     \prescript{AB}{a}\nabla^{(-\alpha,1)}f(t)&=\prescript{AB}{a}\nabla^{-\alpha}f(t); \\
     \prescript{AB}{}\nabla_b^{(-\alpha,1)}f(t)&=\prescript{AB}{}\nabla_b^{-\alpha}f(t).
     \end{align*}
\item When $\alpha=0$, we recover the function $f(t)$ itself: \[\prescript{AB}{a}\nabla^{(0,\mu)}f(t)=\prescript{AB}{}\nabla_b^{(0,\mu)}f(t)=f(t).\]
\item  When $\mu=-1$ we have
     \begin{align}
      \label{r11} \prescript{AB}{a}\nabla^{(-\alpha,-1)}f(t)&=\prescript{ABR}{a}\nabla^{\alpha} f(t); \\
      \label{r12} \prescript{AB}{}\nabla_b^{(-\alpha,-1)}f(t)&=\prescript{ABR}{}\nabla_b^\alpha f(t).
     \end{align}
     More generally, for $\mu=-n$ we have
     \begin{align}
     \label{r21} \prescript{AB}{a}\nabla^{(-\alpha,-n)}f(t)&=\left(\prescript{ABR}{a}\nabla^{\alpha}\right)^n f(t); \\
     \label{r22} \prescript{AB}{}\nabla_b^{(-\alpha,-n)}f(t)&=\left(\prescript{ABR}{}\nabla_b^\alpha\right)^nf(t).
     \end{align}
     For a proof of this, see Theorem \ref{iterated:ABR} below.
\item In the case $\alpha\rightarrow1$ we have the following conventions:
     \begin{align*}
     \prescript{AB}{a}\nabla^{(-1,\mu)}f(t)&=\left(\prescript{}{a}\nabla^{-1}f(t)\right)^\mu=\prescript{}{a}\nabla^{-\mu}f(t); \\
     \prescript{AB}{a}\nabla^{(-1,-\mu)}f(t)&=\prescript{}{a}\nabla^{\mu}f(t).
     \end{align*}
   \end{itemize}
   \end{remark}

\begin{theorem} \label{iterated:ABR}
For any $a<b$ in $\mathbb{R}$, $\alpha\in(0,1)$, and $n\in\mathbb{N}$,  the identities \eqref{r21} and \eqref{r22} are valid.
\end{theorem}

\begin{proof}
We prove the result in the special case $n=1$, i.e. the equations \eqref{r11} and \eqref{r12}. The general result follows from these two equations together with the semigroup property proved in the next section (Theorem \ref{semi}). By symmetry, we consider only \eqref{r11}.

Substituting $\mu=1$ in the representation \eqref{iterated AB diff}, we find:
     \begin{equation*}
        \prescript{AB}{a}\nabla^{(-\alpha,-1)}f(t) =\sum_{k=0}^\infty \frac{\binom{-1}{k}B(\alpha)^1 \alpha^k}{ (1-\alpha)^{1+k}}\prescript{}{a}\nabla^{-k\alpha}f(t).
     \end{equation*}
Since we know that $\binom{-1}{k}=(-1)^k$ and $~_{a}\nabla^{-0}f(t)=f(t)$, this expression becomes:
    \begin{equation*}
    \prescript{AB}{a}\nabla^{(-\alpha,-1)}f(t) = \frac{B(\alpha)}{1-\alpha}\left[f(t)+\sum_{k=1}^\infty \left(\frac{-\alpha}{1-\alpha}\right)^k \left(\prescript{}{a}\nabla^{-\alpha k} f\right)(t)\right].
    \end{equation*}
And Lemma \ref{TQ} tells us that the right-hand side of the final equation is precisely $\prescript{ABR}{a}\nabla^\alpha f(t)$, as required.
\end{proof}

   \begin{remark}
   For the sake of comparisons in this sequel, we invite the reader to check \cite{Thbinomial} about fractional sums and differences with binomial coefficients.
   \end{remark}
   
   \begin{example}
As an illustrative example, we apply the operators introduced in Definitions \ref{Ab iterated} and \ref{rAb iterated} to some simple functions, as follows:
\begin{align*}
\prescript{AB}{a}\nabla^{(-\alpha,\mu)}(t-a)^{\overline{\gamma-1}}&=\sum_{k=0}^\infty \frac{\binom{\mu}{k}(1-\alpha)^{\mu-k} \alpha^k}{B(\alpha)^\mu}\prescript{}{a}\nabla^{-k\alpha}(t-a)^{\overline{\gamma-1}} \\
&=\sum_{k=0}^\infty \frac{\binom{\mu}{k}(1-\alpha)^{\mu-k} \alpha^k}{B(\alpha)^\mu}\cdot\frac{\Gamma(\gamma)}{\Gamma(\gamma+k\alpha)}(t-a)^{\overline{\gamma+k\alpha-1}}; \\
%&=(t-a)^{\overline{\gamma-1}}\left(\frac{1-\alpha}{B(\alpha)}\right)^{\mu}\sum_{k=0}^\infty \frac{\binom{\mu}{k}(1-\alpha)^{\mu-k} \alpha^k}{B(\alpha)^\mu}\frac{\Gamma(\gamma)}{\Gamma(\gamma+k\alpha)}\left(\frac{\alpha}{1-\alpha}\right)^k(t-a)^{\overline{\gamma+k\alpha-1}}; \\
\prescript{AB}{}\nabla_b^{(-\alpha,\mu)}(b-t)^{\overline{\gamma-1}}&=\sum_{k=0}^\infty \frac{\binom{\mu}{k}(1-\alpha)^{\mu-k} \alpha^k}{B(\alpha)^\mu}\nabla_b^{-k\alpha}(b-t)^{\overline{\gamma-1}} \\
&=\sum_{k=0}^\infty \frac{\binom{\mu}{k}(1-\alpha)^{\mu-k} \alpha^k}{B(\alpha)^\mu}\cdot\frac{\Gamma(\gamma)}{\Gamma(\gamma+k\alpha)}(b-t)^{\overline{\gamma+k\alpha-1}}.
\end{align*}
Note that we have used here the following facts \cite[Lemma 3.3]{ThFer}, \cite[Proposition 3.8]{dualR}:
\begin{align*}
\prescript{}{a}\nabla^{-\alpha}(t-a)^{\overline{\gamma-1}}&=\frac{\Gamma(\gamma)}{\Gamma(\gamma+\alpha)}(t-a)^{\overline{\alpha+\gamma-1}}; \\
\nabla_b^{-\alpha}(b-t)^{\overline{\gamma-1}}&=\frac{\Gamma(\gamma)}{\Gamma(\gamma+\alpha)}(b-t)^{\overline{\gamma+\alpha-1}}.
\end{align*}
   \end{example}

\subsection{Fundamental properties}

In this section we prove some important properties of the definition proposed in the previous section, which demonstrate its naturality and usefulness.

\begin{theorem}[Nabla discrete Laplace transforms in the iterated AB model]
Let $\alpha, \mu$ and $f$ be as in Definition \ref{Ab iterated}. Then, we have
\begin{equation}\label{transf}
\left( \mathcal{K}_a~^{AB}_{a}\nabla^{(-\alpha,\mu)}f\right)(z)= \left(\frac{1-\alpha}{B(\alpha)} + \frac{\alpha}{B(\alpha)}z^{-\alpha} \right)^\mu(\mathcal{K}_af)(z).
\end{equation}
\end{theorem}

\begin{proof}
We use the Definition \ref{Ab iterated} of the iterated AB fractional difference, and the Lemma \ref{F} concerning the nabla discrete Laplace transform of fractional difference operators.

\begin{align*}
\left( \mathcal{K}_a\prescript{AB}{a}\nabla^{(-\alpha,\mu)}f\right)(z)&= \sum_{k=0}^\infty \frac{\binom{\mu}{k}(1-\alpha)^{\mu-k} \alpha^k}{B(\alpha)^\mu}(\mathcal{K}_a\prescript{}{a}\nabla^{-k\alpha}f)(z) \\
   &= \sum_{k=0}^\infty \frac{\binom{\mu}{k}(1-\alpha)^{\mu-k} \alpha^k}{B(\alpha)^\mu} z^{-\alpha k}(\mathcal{K}_af)(z) \\
   &=   \sum_{k=0}^\infty \frac{\binom{\mu}{k}(1-\alpha)^{\mu-k} (z^{-\alpha}\alpha)^k}{B(\alpha)^\mu} (\mathcal{K}_af)(z) \\
   &=\left(\frac{1-\alpha}{B(\alpha)} + \frac{\alpha}{B(\alpha)}z^{-\alpha} \right)^\mu(\mathcal{K}_af)(z),
\end{align*}
where in the last step the binomial theorem is applied.
\end{proof}

\begin{theorem}[The semigroup property] \label{semi}
Let $\alpha \in [0,1]$ and $\mu, \nu \in \mathbb{R}$ and $f$ be a function defined on $\mathbb{N}_a$. Then the left and right iterated AB fractional difference
operators each have the semigroup property in $\mu$. Namely, for any $t \in \mathbb{N}_a$ we have
\begin{equation}\label{semiL}
  \prescript{AB}{a}\nabla^{(-\alpha, \mu)}\prescript{AB}{a}\nabla^{(-\alpha, \nu)}f(t)=\prescript{AB}{a}\nabla^{(-\alpha, (\mu+\nu))}f(t),
\end{equation}
and for any $t \in {}_b\mathbb{N}$ we have
\begin{equation}\label{semiR}
  \prescript{AB}{}\nabla_b^{(-\alpha, \mu)}\prescript{AB}{}\nabla_b^{(-\alpha, \nu)}f(t)=\prescript{AB}{}\nabla_b^{(-\alpha, (\mu+\nu))}f(t),
\end{equation}
\end{theorem}

\begin{proof}
Using the fact \eqref{semi nablal} that the left fractional sums have the semigroup property, and by the help of the identity \[\sum_{k=0}^m \binom{\mu}{k}\binom{\nu}{m-k}= \binom{\mu+\nu}{m},\] we have
\begin{align*}
&\prescript{AB}{a}\nabla^{(-\alpha, \mu)}\prescript{AB}{a}\nabla^{(-\alpha, \nu)}f(t) \\
&\hspace{1cm}=\sum_{k=0}^\infty \frac{\binom{\mu}{k}(1-\alpha)^{\mu-k} \alpha^k}{B(\alpha)^\mu}\prescript{}{a}\nabla^{-k\alpha}\left[\sum_{n=0}^\infty \frac{\binom{\nu}{n}(1-\alpha)^{\nu-n} \alpha^n}{B(\alpha)^\nu}\prescript{}{a}\nabla^{-n\alpha}f(t)\right] \\
&\hspace{1cm}=\sum_{k,n}\frac{\binom{\mu}{k}\binom{\nu}{n}(1-\alpha)^{\mu+\nu-k-n} \alpha^{k+n}}{B(\alpha)^{\mu+\nu}}\prescript{}{a}\nabla^{-k\alpha}\prescript{}{a}\nabla^{-n\alpha}f(t) \\
&\hspace{1cm}=\sum_{m=0}^\infty\sum_{k=0}^m\frac{\binom{\mu}{k}\binom{\nu}{m-k}(1-\alpha)^{\mu+\nu-m} \alpha^m}{B(\alpha)^{\mu+\nu}}\prescript{}{a}\nabla^{-m\alpha}f(t) \\
&\hspace{1cm}=\sum_{m=0}^\infty\frac{\binom{\mu+\nu}{m}(1-\alpha)^{\mu+\nu-m} \alpha^m}{B(\alpha)^{\mu+\nu}}\prescript{}{a}\nabla^{-m\alpha}f(t)=\prescript{AB}{a}\nabla^{(-\alpha, \mu+\nu)}f(t).
\end{align*}
This proves the left semigroup property \eqref{semiL}. The proof for \eqref{semiR} is similar, starting from \eqref{semi nablar} and using the same binomial identity.
\end{proof}

\begin{theorem}[Integration by parts] \label{IbP}
Let $\alpha\in[0,1]$ and $\mu\in\mathbb{R}$ and $a,b\in\mathbb{R}$ with $a\equiv b$ modulo $1$. For any function $f$ defined on $\mathbb{N}_a$ and $g$ defined on ${}_b\mathbb{N}$, we have the following integration by parts identity:
\begin{equation}
\sum_{s=a+1}^{b-1}g(s)\prescript{AB}{a}\nabla^{(-\alpha,\mu)}f(s)=\sum_{s=a+1}^{b-1}f(s)\prescript{AB}{}\nabla_b^{(-\alpha,\mu)}g(s).
\end{equation}
\end{theorem}

\begin{proof}
We know from \cite[Theorem 4.1]{ThFer} the following integration by parts identity for standard fractional difference-sum operators:
\begin{equation}
\label{IbP:orig}
\sum_{s=a+1}^{b-1}g(s)\prescript{}{a}\nabla^{-\alpha}f(s)=\sum_{s=a+1}^{b-1}f(s)\nabla_b^{-\alpha}g(s).
\end{equation}
Using respectively Definition \ref{Ab iterated}, equation , and Definition \ref{rAb iterated}, we have
\begin{align*}
\sum_{s=a+1}^{b-1}g(s)\prescript{AB}{a}\nabla^{(-\alpha,\mu)}f(s)&=\sum_{s=a+1}^{b-1}g(s)\sum_{k=0}^\infty \frac{\binom{\mu}{k}(1-\alpha)^{\mu-k} \alpha^k}{B(\alpha)^\mu}\prescript{}{a}\nabla^{-k\alpha}f(s) \\
&=\sum_{k=0}^\infty \frac{\binom{\mu}{k}(1-\alpha)^{\mu-k} \alpha^k}{B(\alpha)^\mu}\sum_{s=a+1}^{b-1}g(s)\prescript{}{a}\nabla^{-k\alpha}f(s) \\
&=\sum_{k=0}^\infty \frac{\binom{\mu}{k}(1-\alpha)^{\mu-k} \alpha^k}{B(\alpha)^\mu}\sum_{s=a+1}^{b-1}f(s)\nabla_b^{-k\alpha}g(s) \\
&=\sum_{s=a+1}^{b-1}f(s)\sum_{k=0}^\infty \frac{\binom{\mu}{k}(1-\alpha)^{\mu-k} \alpha^k}{B(\alpha)^\mu}\nabla_b^{-k\alpha}g(s) \\
&=\sum_{s=a+1}^{b-1}f(s)\prescript{AB}{}\nabla_b^{(-\alpha,\mu)}g(s).
\end{align*}
\end{proof}

Integration by parts identities as in Theorem \ref{IbP} are the main tool used to study discrete variational problems in the frame of iterated AB difference-sums \cite{TD ADE 2016,Thabetnew}.

\section{Fractional difference equations and applications} \label{s:3}

 Let $\alpha \in (0,1]$ and $\mu  \in \mathbb{R}^+$. Consider a general  fractional ordinary difference equation of the form
 \begin{equation}\label{eqt}
   \prescript{AB}{0}\nabla^{(-\alpha, -\mu)}x(t)= -A x(t)+b(t),~~A \in \mathbb{R}^+.
 \end{equation}

We search for a series solution, i.e. one of the form \[x(t)=\sum_{s=0}^\infty c_s t^{\overline{\alpha s}},\] where we assume the given function $b$ can be written as \[b(t)=\sum_{s=0}^\infty b_s ~t^{\overline{\alpha s}}.\]
 First, we substitute $x(t)$ into the left-hand side of \eqref{eqt} and utilise Lemma \ref{power nabla left and right} to get

 \begin{align*}
 \nonumber \prescript{AB}{0}\nabla^{(-\alpha, -\mu)}x(t) &= \sum_{k=0}^\infty \frac{\binom{-\mu}{k}
    (1-\alpha)^{-\mu-k}}{B(\alpha)^{-\mu}}\prescript{}{0}\nabla^{-\alpha k}\left(\sum_{i=0}^\infty c_i t^{\overline{i \alpha}}\right) \\
 \nonumber  &= \sum_{k=0}^\infty\sum_{i=0}^\infty \frac{\binom{-\mu}{k}
    (1-\alpha)^{-\mu-k}}{B(\alpha)^{-\mu}}~c_i\frac{\Gamma(i\alpha+1)}{\Gamma((k +i)\alpha+1)}t^{\overline{(i+k)\alpha}}\\
    \begin{split}
    &=  \sum_{m=0}^\infty \frac{t^{\overline{m \alpha}}}{B(\alpha)^{-\mu} \Gamma(m \alpha+1)}\sum_{k=0}^m \\
    &\hspace{3cm} c_{m-k}\binom{-\mu}{k}(1-\alpha)^{-\mu-k}\Gamma((m-k)\alpha+1),
    \end{split}
 \end{align*}
 where we have set $m=k+i$ in the last line. On the other hand, the right hand side of \eqref{eqt} can be written as \[\sum_{m=0}^\infty [-A c_m+b_m]t^{\overline{\alpha m}}.\]
Now, if we equate the coefficients in these two infinite series, we reach the identity
\begin{equation}\label{id1}
  \frac{1}{\Gamma(m \alpha+1)} \sum_{k=0}^m \frac{c_{m-k} \binom{-\mu}{k} B(\alpha)^\mu \alpha^k \Gamma((m-k)\alpha+1)}  {(1-\alpha)^{\mu+k}}=-Ac_m+b_m,\quad m\in\mathbb{N}_0.
\end{equation}
Solving for $m=0$, we have \[c_0=\frac{b_0}{A+\left(\frac{B(\alpha)}{1-\alpha}\right)^\mu}.\]
For positive $m$, the identity \eqref{id1} will yield
\begin{equation}\label{id2}
  c_m= \frac{b_m}{A +\left(\frac{B(\alpha)}{1-\alpha}\right)^\mu}-\sum_{k=1}^\infty \frac{c_{m-k} \binom{-\mu}{k} \alpha^k B(\alpha)^\mu  \Gamma((m-k)\alpha+1)}   {(1-\alpha)^{\mu+k} \Gamma(m\alpha+1)\Gamma(m\alpha+1)\left(A+\left(\frac{B(\alpha)}{1-\alpha}\right)^\mu \right)}.
\end{equation}
Hence, we obtain the following solution:
\begin{equation}\label{sol}
  x(t)=\frac{b(t)}{A+\left(\frac{B(\alpha)}{1-\alpha}\right)^\mu}-\sum_{m=1}^\infty  t^{\overline{m \alpha}} \sum_{k=1}^m \frac{c_{m-k} \binom{-\mu}{k}\alpha^kB(\alpha)^\mu \Gamma((m-k)\alpha+1)}
  {(1-\alpha)^{\mu+k}\Gamma(m\alpha+1)\left(A+\left(\frac{B(\alpha)}{1-\alpha}\right)^\mu\right)}.
\end{equation}
Actually, the coefficients $c_i,~i=1,2,3,...$ can be calculated recursively from \eqref{id2}.

\begin{remark}
Of special interest is the particular case $\mu=1$ of the above problem. In this case, the solution representation \eqref{sol} becomes
\[x(t)=\frac{b(t)}{A+\left(\frac{B(\alpha)}{1-\alpha}\right)}-\sum_{m=1}^\infty  t^{\overline{m \alpha}} \sum_{k=1}^m \frac{c_{m-k} (-1)^k\alpha^kB(\alpha) \Gamma((m-k)\alpha+1)}
  {(1-\alpha)^{1+k}\Gamma(m\alpha+1)\left(A+\frac{B(\alpha)}{1-\alpha}\right)},\]
this being a solution of the fractional difference equation
\[\prescript{ABR}{0}\nabla^\alpha x(t)=-Ax(t)+b(t),\] where $b(t)=\sum_{s=0}^\infty b_s ~t^{\overline{\alpha s}}$ and $c_0=\frac{b_0}{A+(\frac{B(\alpha)}{1-\alpha})}$ and the coefficients $c_i$, $i\in\mathbb{N},$ can be determined from (\ref{id2}) with $\mu=1$.
\end{remark}

\begin{remark}
It is worth noting that the semigroup property of Theorem \ref{semi} will be invaluable in the further study of fractional difference equations in the discrete iterated AB model. There are many classes of difference equations %, in both the discrete (difference) and continuous (differential) cases, 
which are easy to solve when we have a semigroup property but difficult or impossible when we do not \cite{Goodrich,dualCaputo,dualR}. Equipped with a semigroup property, we can easily cancel operators, in order to simplify and solve an equation, by applying new difference-sum operators to both left and right sides of the equation.
\end{remark}

\section{Conclusions} \label{s:4}

In this paper, we have introduced a new kind of discrete fractional calculus, whose advantages over existing models can be summarised as follows.
\begin{itemize}
\item \textbf{Atangana--Baleanu} fractional calculus, in the continuous case, has discovered many applications in modelling fractional systems with non-local and non-singular dynamics, behaviour that cannot be modelled using the classical Riemann--Liouville kernels. In a short space of time, the AB formula has established itself as a major competitor among the many different approaches to fractional calculus.

\item \textbf{Discrete} fractional calculus (DFC) is a whole different field of research from continuous fractional calculus. Like the original discrete calculus and difference equations, DFC can be used to study many real-world processes whose behaviour is too discrete to be well modelled by continuous fractional calculus.

\item In any type of calculus -- discrete or continuous, integer-order or fractional -- an important question is whether or not a \textbf{semigroup property} is satisfied. If we apply the operator twice with order $\alpha$, do we get the same result as by applying it once with order $2\alpha$? This fundamental issue has given rise to much debate about the validity of certain fractional models, and several new models have been proposed purely in order to regain a semigroup property.
\end{itemize}

Our definition gives a unique way of combining the particular structure of the AB formula, the discrete behaviour of DFC, and the semigroup property thanks to the introduction of a second parameter. Previous literature covers the combination of any two of these three (the discrete AB model, the continuous iterated AB model, and discrete models with semigroup properties), but this is the first time that all three have been put together.

In this paper, we covered some basic properties and examples of our new fractional difference-sum operator. By analysing the effect on it of the discrete Laplace transform, we also demonstrated how it can be used to solve a certain family of fractional difference equations.

Numerical methods have been an important part of the recent development of fractional calculus, including in the Atangana--Baleanu model \cite{djida-atangana-area,owolabi,yadav-pandey-shukla,ucar-ucar-ozdemir-hammouch}. Discrete calculus is often better suited to numerical schemes than the continuous case, due to the finite structure of computation there \cite{lakshmikantham-trigiante}. Thus, we expect that any new progress in discrete fractional calculus of AB type should have ramifications in numerical analysis, although to explore such ramifications would be beyond the scope of the current paper.

\section*{Acknowledgements}
The first author would like to thank Prince Sultan University for funding this work through research group Nonlinear Analysis Methods in Applied Mathematics (NAMAM)  group number RG-DES-2017-01-17. The second author would like to thank the Engineering and Physical Sciences Research Council (EPSRC) for their support in the form of a research student grant.

%\nocite{*}
%\bibliography{aipsamp}% Produces the bibliography via BibTeX.

\end{document}